\documentclass[11pt]{article}
\usepackage[T1]{fontenc}
\usepackage{lmodern}
\usepackage{amsmath,amssymb,amsthm,mathtools,mathrsfs}
\usepackage[a4paper,margin=28mm]{geometry}
\usepackage{microtype}
\usepackage[hidelinks]{hyperref}
\usepackage{tikz}

\newtheorem{theorem}{Theorem}
\newtheorem{lemma}[theorem]{Lemma}
\newtheorem{conjecture}[theorem]{Conjecture}
\newtheorem{proposition}[theorem]{Proposition}
\theoremstyle{definition}
\newtheorem{remark}[theorem]{Remark}
\newtheorem{definition}[theorem]{Definition}

\newcommand{\F}{\mathcal F}
\newcommand{\G}{\mathcal G}
\newcommand{\A}{\mathcal A}
\newcommand{\B}{\mathcal B}
\newcommand{\V}{\mathcal V}
\newcommand{\eps}{\varepsilon}
\newcommand{\Pp}{\mathcal P}

\newcommand{\ar}{\mathop{}\!\mathrm{ar_m}}
\newcommand{\dist}{\operatorname{dist}}
\newcommand{\HH}{\operatorname{H}}
\newcommand{\Hh}{\mathcal H}
\newcommand{\Forb}{\operatorname{Forb}}
\newcommand{\LAw}{\mathop{}\!\mathrm{La}}
\newcommand{\La}{\mathop{}\!\mathrm{La}}

\title{Embedding rooted blow-ups of tree posets}
\author{Bal\'azs Patk\'os\thanks{HUN-REN Alfr\'ed R\'enyi Institute of Mathematics and Department of Computer Science and Information Theory, Budapest University of Technology and Economics, email: patkos@renyi.hu}}
\date{}

\begin{document}
\maketitle

\begin{abstract}
    A tree poset $T$ is a poset whose Hasse diagram is a tree. Bukh proved that if a family $\F\subseteq 2^{[n]}$ contains $(h(T)-1+\varepsilon)\binom{n}{\lfloor \frac{n}{2}\rfloor}$ sets, then $\F$ contains a weak copy of $T$, where $h(T)$ is the height of $T$, the number of elements in a longest chain of $T$. Several strengthenings and generalizations of this result have been obtained. We prove the following robust variant. For a tree poset $T$ and $x\in T$, the $b$-blow-up $T(x,b)$ rooted at $x$ is the tree poset that we obtain from $T$ by replacing every element $u$ by $b^d$ new elements, where $d$ is the distance $d(x,u)$ in the Hasse diagram of $T$ and an edge $uv$ with $v$ being closer to $x$ is replaced by edges such that every new copy of $v$ is joined to $b$ new copies of $u$ such that these new copies form pairwise disjoint sets for the copies of $v$. We prove that for any tree poset $T$, $x\in T$, and $\varepsilon>0$ there exists $\delta$ such that if $\F\subseteq 2^{[n]}$ contains $(h(T)-1+\varepsilon)\binom{n}{\lfloor \frac{n}{2}\rfloor}$ sets, then $\F$ contains a weak copy of $T(x,\lfloor \delta n\rfloor)$. This settles a conjecture of Treglown and the author. As applications, we derive the known asymptotic counting and random versions of Bukh's theorem from this stronger embedding result, and obtain new maximal anti-Ramsey results for tree posets.
\end{abstract}

\section{Introduction}

We use standard notation: $[n]$ denotes the set of the first $n$ positive integers, and for any set $S$, we write $2^{S}$ and $\binom{S}{k}$ to denote its power set and the family of its $k$-element subsets, respectively.

\medskip

One of the starting points of extremal finite set theory is the theorem of Sperner stating that if $\F\subseteq 2^{[n]}$ does not contain $F,F'$ with $F\subseteq F'$, then $|\F|\le \binom{n}{\lfloor \frac{n}{2}\rfloor}$ and equality holds if and only if $\F=\binom{[n]}{\lfloor \frac{n}{2}\rfloor}$ or $\F=\binom{[n]}{\lceil \frac{n}{2}\rceil}$. To generalize
 forbidden containment patterns, Katona and Tarj\'an \cite{KatonaTarjan} introduced the following definition: a family $\G$ of sets is a \textit{weak copy} of the poset $(P,\preceq)$, if there exists a bijection $\iota:P\rightarrow \G$ such that $p\preceq q$ implies $\iota(p) \subseteq \iota(q)$. We say that $\F$ is $P$-free if it does not contain any weak copies of $P$. Determining $\LAw(n,P)=\max \{|\F|: \F\subseteq 2^{[n]}, ~\F ~\text{is $P$-free}\}$ is an extensively studied extremal problem, see \cite{AMP,GLsurv} and Chapter 7 of \cite{GP} for surveys.

 A natural way to construct a large $P$-free family is to consider as many middle layers of $2^{[n]}$ as possible. Formally, let $e(P)$ denote the maximum integer $k$ such that $\cup_{i=1}^k\binom{[n]}{\lfloor \frac{n-k}{2}\rfloor +i}$ is $P$-free for all $n$. Then by definition, we have $\LAw(n,P)\ge (e(P)+o(1))\binom{n}{\lfloor \frac{n}{2}\rfloor}$. It was conjectured that for posets $P$, it holds that $e(P)=\lim_{n\rightarrow \infty}\frac{\LAw(n,P)}{\binom{n}{\lfloor \frac{n}{2}\rfloor}}=:\pi(P)$. (It is not known whether the limit $\pi(P)$ exists for all posets, so $\pi^-(P)$ and $\pi^+(P)$ are used for liminf and limsup, respectively.)

 This conjecture was refuted by Ellis, Ivan, and Leader \cite{EIL}, so it is interesting to see how far one can push the extremal result for posets with the property $e(P)=\pi(P)$. Supersaturation, counting results and probabilistic versions of the extremal problem were obtained in \cite{BGW,GNPV,Jetal,PT}. The most important class of posets for which $e(P)=\pi(P)$ was proved \cite{Bukh} is that of tree posets.

The \textit{directed Hasse diagram} $\overrightarrow{H}(P)$ of $(P,\preceq)$ is a directed graph with vertex set $P$ and $\overrightarrow{pq}$ is an arc if and only if $p\preceq q$ and there does not exist $z\neq p,q$ with $p\preceq z\preceq q$. The \textit{undirected Hasse diagram} $H(P)$ is
obtained from $\overrightarrow{H}(P)$ by dropping the orientation of arcs.
A \emph{tree poset} is a finite nonempty poset whose undirected Hasse
diagram is a tree. Its height $h(T)$ is the maximum number of elements
in a chain. Note that for any tree poset $T$, $e(T)=h(T)-1$ holds.

\begin{theorem}[Bukh, \cite{Bukh}]\label{tree}
    For any tree poset $T$, we have $\pi(T)=h(T)-1$.
\end{theorem}

As mentioned earlier, Theorem \ref{tree} was strengthened in several directions. One important tool in some of the proofs of these results was to show that a family of size $(h(T)-1+\varepsilon)\binom{n}{\lfloor \frac{n}{2}\rfloor}$ contains not only a copy of $T$, but a robust version of $T$. Here is the formal definition introduced in \cite{PT}.

\begin{definition}
    Let $T$ be a tree poset and $x\in T$, and let $d\ge 2$ be a positive integer. The \textit{$d$-blow-up $T(x,d)$ rooted at $x$} is the tree poset whose Hasse diagram is defined as follows: for each $u \in T$, if $u$
    is at distance $\rho$ from $x$ in  the undirected Hasse diagram of $T$, then $u$ is replaced with $d^{\rho}$ elements $u^1,u^2,\dots, u^{d^{\rho}}$; furthermore, if $uv$ is an edge of the Hasse diagram of $T$ with $v$ being at distance $\rho-1$ from $x$ in $T$, then the $u^i$s are partitioned into $d^{\rho-1}$ pairwise disjoint sets $U^1,U^2,\dots, U^{d^{\rho-1}}$ each of size $d$, and for every $j \in [d^{\rho-1}]$,
    $v^j$ is joined to all members of $U^j$. The orientation of all such edges is the same as that of $uv$.
\end{definition}

The authors of \cite{PT} formulated a general conjecture that they proved for a special class of tree posets. Their results were extended to a different subclass of tree posets in \cite{LPW}.

\begin{conjecture}[Conjecture ~1.3 in \cite{PT}]\label{conj}
For every tree poset $T$ and every $\eps>0$, there exist $\delta>0$ and
$n_0$, depending only on $T$ and $\eps$, such that the following holds for
every $n\ge n_0$ and every $x\in T$. If
$\F\subseteq 2^{[n]}$ with $|\F|\ge\bigl(h(T)-1+\eps\bigr)\binom n{\lfloor n/2\rfloor}$,
then $\F$ contains a copy of $T(x,\lfloor\delta n\rfloor)$.
\end{conjecture}

Our main result confirms Conjecture \ref{conj}.

\begin{theorem}\label{thm:main}
    Conjecture ~\ref{conj} is true.
\end{theorem}

The proofs of tree poset results fall into two categories. Some of them \cite{BGW,BJ,Jetal,P1} follow and strengthen the original ideas of Bukh \cite{Bukh}: define a sequence $T=T_0\supset T_1\supset \dots \supset T_\ell$ of tree posets such that $T_\ell$ and $T_{j-1}\setminus T_j$ are chains for all $1\le j\le \ell$, and then embed $T_\ell$ first and then keep extending the embedding by adding the \textit{chain} corresponding to $T_{j-1}\setminus T_j$. Other proofs \cite{GNPV,LPW,PT} (that so far worked only for special classes of tree posets) used a "graph theoretic" approach: they extended the embedding by always adding an \textit{edge}. Our proof uses this second approach.

\medskip

The \textit{Lubell mass} of a family is defined as
\[
 \lambda_n(\A)=\sum_{A\in\A}\binom n{|A|}^{-1}.
\]
If $\mathscr C$ is a uniformly random maximal chain in $2^{[n]}$, then
$\lambda_n(\A)=\mathbb E|\A\cap\mathscr C|$ and
$\lambda_n(\A)\ge |\A|/\binom{n}{\lfloor n/2\rfloor}$. For sets $A,B\subseteq[n]$, write
$\dist(A,B)=|A\triangle B|$ for the Hamming distance of $A$ and $B$. Set $\mathcal B_R(A)=\{B\subseteq[n]:\dist(A,B)\le R\}$.
All asymptotic statements are as $n\to\infty$, with every parameter
other than $n$ fixed unless stated otherwise.

\section{Proof of Theorem \ref{thm:main}}

Let us describe briefly the main steps of the proof before getting into the details. The first step, Proposition \ref{prop:extraction}, is the completely new element, the other two steps, Lemma \ref{lem:bounded} and Lemma \ref{lem:gap} use the ideas of \cite{LPW}. A \textit{bipartite containment graph} $B=B(\A,\B)$ has vertex set $\A\cup \B$ with $A,B$ forming an edge if and only if $A\subseteq B$. Note that the definition is not symmetric, so $B(\A,\B)\neq B(\B,\A)$. Proposition \ref{prop:extraction} states that a family $\F$ of size $(h+\varepsilon)\binom{n}{\lfloor \frac{n}{2}\rfloor}$ contains subfamilies $\F_0,\F_1,\dots,\F_h$ such that the minimum degree in all $B_i=B(\F_{i-1},\F_i)$ is large. Previous proofs \cite{GNPV,LPW} could only establish that the minimum degree is large in all $\A$ parts or in all $\B$ parts and so those proofs only worked for \textit{monotone} tree posets (tree posets with a smallest or a largest element).

The second step, Lemma \ref{lem:bounded}, settles the case when in all graphs $B(\F_{i-1},\F_i)$ the edges correspond to pairs $A\subseteq B$ with bounded $|B\setminus A|$. Then Lemma \ref{lem:gap} handles the case where for some $i$ $B(\F_{i-1},\F_i)$ has large set differences on its edges. As mentioned before, the ideas of these lemmas are from \cite{LPW}.

\begin{proposition}\label{prop:extraction}
Fix integers $h\ge1$, $Z\ge2$, and a real $\eps>0$. There exist $c>0$
and $n_0$, depending only on $h,Z,\eps$, such that for $n\ge n_0$ every
family $\F\subseteq2^{[n]}$ with $|\F|\ge(h+\eps)\binom{n}{\lfloor n/2\rfloor}$ contains
nonempty, pairwise disjoint families
$\V_0,\ldots,\V_h$
and bipartite containment graphs $B_i=B(\V_{i-1},\V_i)$,
for $1\le i\le h$, with the following properties. Each graph has a type
$t_i\in\{1,\ldots,Z,+\}$, and:
\begin{enumerate}
 \item if $t_i=j\le Z$, all edges have size difference exactly $j$,
 and $B_i$ has minimum degree at least $cn^j$;
 \item if $t_i=+$, all edges have size difference greater than $Z$,
 and $B_i$ has minimum degree at least $cn^{Z-1}$.
\end{enumerate}
\end{proposition}

\begin{proof}
Let
$\A=\{F\in \F: ||F|-n/2|<n^{2/3}\}$.
The binomial tail bound shows that  $|\F\setminus \A|\le |\{G\subseteq [n]:||G|-n/2|\ge n^{2/3}\}|=o(\binom{n}{\lfloor n/2\rfloor})$. Thus, for sufficiently large $n$,
\begin{equation}\label{eq:lambda}
 |\A|\ge(h+\frac{\varepsilon}{2})\binom{n}{\lfloor n/2\rfloor},\qquad
 \lambda:=\lambda_n(\A)\ge h+\frac{\varepsilon}{2}.
\end{equation}

Let $m=\lfloor n/3\rfloor$ and define
$D_j=\binom mj$ $(1\le j\le Z)$,
 $D_+=\frac1{n^2}\binom m{Z+1}$.
For every strict containment $A\subset B$ in $\A$, define its type by
\[
 \tau(A,B)=
 \begin{cases}
  |B|-|A|,&|B|-|A|\le Z,\\
  +,&|B|-|A|>Z.
 \end{cases}
\]

\medskip\noindent
Let $A,B\in\A$ be distinct comparable sets.
For a uniformly random maximal chain $\mathscr C$,
\[\ \Pr(B\in\mathscr C\mid A\in\mathscr C)=
 \begin{cases}
  \displaystyle\binom{n-|A|}{|B\setminus A|}^{-1},&A\subset B,\\[5pt]
  \displaystyle\binom{|A|}{|A\setminus B|}^{-1},&B\subset A.
 \end{cases}
\]The upper binomial parameter is at least $m$, while
$|A\triangle B|\le2n^{2/3}<m/2$ for sufficiently large $n$. Monotonicity of binomial
coefficients therefore gives, if the edge has type $t$,
\begin{equation*}\label{eq:conditional-bound}
 \Pr(B\in\mathscr C\mid A\in\mathscr C)\le D_t^{-1}.
\end{equation*}
For type $+$ the stronger bound
$\binom m{Z+1}^{-1}=(n^2D_+)^{-1}$ holds. In particular, if for some real $\theta$, fewer than
$\theta D_t$ edges of one type $t$ are assigned to $A$, their total
expected contribution to the number of assigned edges with both ends
on $\mathscr C$ is at most
\begin{equation}\label{eq:vertex-cost}
 \theta\Pr(A\in\mathscr C).
\end{equation}

\medskip\noindent
Color every member of $\A$ independently and uniformly with 
colors $0,\ldots,h$. For a maximal chain $\mathscr C$, list its members
of $\A$ in increasing order as
$A_1\subset\cdots\subset A_X,~ X=|\A\cap\mathscr C|$.
Partition this list into consecutive blocks of length $h+1$, ignoring
the final incomplete block. A block is \emph{good} if its colors, in
increasing order, are exactly $0,1,\ldots,h$.

Put $p=(h+1)^{-h-1}$. For every fixed maximal chain, each block is good with
probability $p$. If $g_\chi(\mathscr C)$ denotes the number of good
blocks under a coloring $\chi$, then
\begin{align}
 \mathbb E_\chi\mathbb E_{\mathscr C}g_\chi(\mathscr C)
 =p\,\mathbb E_{\mathscr C}\left\lfloor\frac{X}{h+1}\right\rfloor
 \notag\ge\frac{p}{h+1}\bigl(\mathbb E_{\mathscr C}X-h\bigr)
 =\frac{p}{h+1}(\lambda-h).\label{eq:color-average}
\end{align}
Here $\lfloor\frac{X}{h+1}\rfloor\ge \frac{X-h}{h+1}$ for every nonnegative integer $X$.
Fix a coloring $\chi$ attaining at least the average, so that
\begin{equation}\label{eq:good-lower}
 \mathbb E_{\mathscr C}g_\chi(\mathscr C)\ge\frac{p}{h+1}(\lambda-h).
\end{equation}
We aim to choose $\V_i$ from the $i$th color class.

\medskip\noindent
There are $W=(Z+1)^h$ type sequences
$\boldsymbol t=(t_1,\ldots,t_h)$. Choose the positive constant
$\theta=\frac{p\frac{\varepsilon}{2}}{2(h+1)W(h+\frac{\varepsilon}{2})}$.
For each sequence, start a fresh pruning process on all the initial
color classes. Between colors $i-1$ and $i$, use all containments of
type $t_i$. Repeatedly delete any set having fewer than
$\theta D_{t_i}$ current neighbors in one of its required adjacent
parts. A set in color class $i$ with $1\le i\le h-1$ has a requirement on each side; a set in color class 0 or $h$
has only its one adjacent requirement.

At each deletion, choose one deficient direction and record all edges
from the deleted set to its currently surviving neighbors in that
direction. Assign those recorded edges to the deleted set. Record
only this one direction, even if both are deficient. All choices are
made with the fixed coloring; the deletion order and recorded edge sets
do not depend on the maximal chain used in the expectations.

If a process terminates with any surviving sets, there are survivors
in every part: each survivor has a positive number of neighbors in every
required adjacent part, and this propagates along the path of parts.
Its surviving graphs have minimum degrees at least $\theta D_{t_i}$
at both ends. These give the required families $\V_0,\V_1,\dots,\V_h$ by the definition of $D_j$s and by $D_j=\Theta(n^j)$ and $D_+=\Theta(n^{Z-1})$.

Suppose, for a contradiction, that every process deletes all $\A$.
Let $E_{\boldsymbol t}$ be the recorded edges for one sequence, and let
$E=\bigcup_{\boldsymbol t}E_{\boldsymbol t}$. For any edge set $J$ put
\[
 N_J(\mathscr C)=|\{\{A,B\}\in J:A,B\in\mathscr C\}|.
\]
In a fixed process each set is deleted once, recording fewer than
$\theta D_t$ edges of one type. Each recorded edge is assigned to the
endpoint deleted when it was recorded. Summing
\eqref{eq:vertex-cost} over the sets of $\A$ yields
\[
 \mathbb E_{\mathscr C}N_{E_{\boldsymbol t}}(\mathscr C)
 \le\theta\sum_{A\in\A}\Pr(A\in\mathscr C)=\theta\lambda.
\]
An edge may occur in several sequences, which only increases the sum
on the right in the following union bound:
\begin{equation}\label{eq:recorded-upper}
 \mathbb E_{\mathscr C}N_E(\mathscr C)\le W\theta\lambda.
\end{equation}

\medskip\noindent
Consider a chain $A_0\subset A_1\subset\cdots\subset A_h$ with colors
$0,1,\ldots,h$, respectively. Its consecutive edges determine the
sequence $t_i=\tau(A_{i-1},A_i)$. In the process for this sequence,
consider the first of these $h+1$ sets to be deleted. All other
sets of the chain are still present. In the chosen deficient
direction the chain has an adjacent set, and their edge is therefore
recorded. Hence every such colored chain has at least one 
edge in $E$ that connects two of its consecutive sets.

In particular, every good block on $\mathscr C$ contains a recorded
edge. Different blocks live on disjoint subchains of $\mathscr C$, so these edges are
distinct. Thus $g_\chi(\mathscr C)\le N_E(\mathscr C)$ for every maximal
chain. By \eqref{eq:good-lower} and \eqref{eq:recorded-upper},
\begin{equation}\label{eq:contradiction}
 \frac{p}{h+1}(\lambda-h)\le W\theta\lambda.
\end{equation}
But \eqref{eq:lambda} implies
$\lambda-h\ge\frac{\varepsilon}{2}\lambda/(h+\frac{\varepsilon}{2})$, and consequently the left side
of \eqref{eq:contradiction} is at least
\[
 \frac{p\frac{\varepsilon}{2}}{(h+1)(h+\frac{\varepsilon}{2})}\lambda=2W\theta\lambda,
\]
contrary to \eqref{eq:contradiction}.

Some type sequence therefore has a nonempty surviving system. Its parts
are disjoint because they belong to different color classes. For large
$n$ and $1\le j\le Z+1$,
$\binom mj\ge(n/4)^j/j!$. Thus a common admissible constant in the
statement is, for example,
\[
 c=\frac{\theta}{4^{Z+1}(Z+1)!}.
\]
This constant depends only on $h,Z,\eps$.
\end{proof}

\begin{remark}\label{rem}
    The case $h=1$ of Proposition \ref{prop:extraction} is implicitly  proved in \cite{LPW} in a much simpler way. The case $h=2$ can also be proved along those lines: for a family $\F\subseteq \{G\subseteq [n]:||G|-n/2|\le n^{2/3}\}$ with $|\F|\ge (2+\varepsilon)\binom{n}{\lfloor \frac{n}{2}\rfloor}$ let us define $\F^2:=\{F\in \F: \lambda_{|F|}(\F_{F,-})\ge 1+\varepsilon/3\}$ and $\F^0:=\{F\in \F:\lambda_{n-|F|}(\F_{F,+})\ge 1+\varepsilon/3\}$, where $\F_{F,-}=\{F'\in \F:F'\subseteq F\}$ and $\F_{F,+}=\{F'\setminus F: F\subseteq F'\in \F\}$. A standard argument from \cite{GNPV} shows that $\lambda_n(\F\setminus \F^2),\lambda_n(\F\setminus \F^0)\le (1+\varepsilon/3)$ and thus $|\F^2|,|\F^0|\ge (1+2\varepsilon/3)\binom{n}{\lfloor \frac{n}{2}\rfloor}$. Therefore $\F^1=\F^2\cap \F^0$ has size at least $\frac{\varepsilon}{3}\binom{n}{\lfloor \frac{n}{2}\rfloor}$. All degrees of $\F^1$ to \textit{both} $\F^0$ and  $\F^2$ are high by the Lubell mass conditions, and the standard graph theory proof yields the desired subgraphs $B_1$ and $B_2$.
\end{remark}

\medskip

A finite poset $Q$ is \emph{saturated} if every maximal chain has
$h(Q)$ elements. For a saturated poset of height $t+1$, write
$Q_0,\ldots,Q_t$ for its natural levels, i.e. $Q_i$ is the set of minimal elements in $Q\setminus \cup_{j=0}^{i-1}Q_j$, and
let $\ell(u)=i$ for $u\in Q_i$. 
As a result of saturation, every Hasse edge joins consecutive
levels and is oriented upward. We use this notation for saturated tree
posets throughout the rest of the proof.

\begin{lemma}\label{lem:bounded}
Let $Q$ be a fixed saturated tree poset of height $t+1$, with levels
$Q_0,\ldots,Q_t$, and fix a root $y\in Q$. Let
$\V_0,\ldots,\V_t\subseteq2^{[n]}$ be nonempty, pairwise disjoint
families. For each $1\le i\le t$, fix a positive integer $j_i$ and a
bipartite graph $B_i=B(\V_{i-1},\V_i)$, all of whose edges
are containments $A\subset B$ with $|B|-|A|=j_i$. Suppose $B_i$ has
minimum degree at least $cn^{j_i}$, where $c>0$ is fixed, and give each edge of $B_i$ the weight $j_i$; thus
$w(uv)=j_{\max\{\ell(u),\ell(v)\}}$.

There are $\beta>0$ and $n_0$, depending only on the fixed data, such
that for $n\ge n_0$ and any prescribed $A\in\V_{\ell(y)}$, the
blow-up $Q(y,\lfloor\beta n\rfloor)$ has an injective order-preserving
embedding sending its root to $A$, with level $i$ represented in $\V_i$
and every edge between consecutive levels embedded using $B_i$.

Moreover, put
\[
 R=\max_{v\in Q}\sum_{e\in P(y,v)}w(e),
\]
where $P(y,v)$ is the unique root-to-$v$ path. The embedding can be
chosen in $\mathcal B_R(A)$, with pairwise disjoint sets of changed
coordinates along each root-to-vertex path. The same conclusion holds
for every integer $1\le b\le\beta n$ in place of
$\lfloor\beta n\rfloor$.
\end{lemma}

\begin{proof}
The case $|Q|=1$ is immediate. Suppose $q=|Q|\ge2$. Color every set in
$\mathcal B_R(A)\setminus\{A\}$ independently and uniformly with the
$q$ original vertices of $Q$, and give $A$ color $y$. 

Regard an edge as a rooted edge $uv$, with $u$ the parent, regardless
of its containment direction. Put $i=\max\{\ell(u),\ell(v)\}$ and
$j=j_i=w(uv)$. Thus $u,v$ lie on consecutive levels and their edge
in an embedding should use $b\le \lfloor \beta n\rfloor$ edges of $B_i$. For any $P\in\V_{\ell(u)}$ with $\dist(P,A)\le R-j$,
all its $B_i$-neighbors in $\V_{\ell(v)}$ lie in $\mathcal B_R(A)$.
Apart from possibly $A$,
their auxiliary colors are independent and uniform. Since there are
at least $cn^j$ neighbors, Chernoff's inequality gives probability
$e^{-\Omega(n^j)}\le e^{-\Omega(n)}$ of having fewer than $cn^j/(4q)$
neighbors of color $v$. There are only $O(n^R)$ sets in $\mathcal B_R(A)$ and
constantly many rooted edges. A union bound gives a coloring satisfying
all these requirements simultaneously.

For an original vertex $u\in Q$, put
\[
 s(u)=\sum_{e\in P(y,u)}w(e).
\]
We now start embedding $Q(y,b)$. We fix an ordering $y=u_1,u_2,\dots,u_q$ of $Q$ such that vertices of $P_{y,u_i}$ precede $u_i$ for all $i$. We first place $y$ to $A$ and then embed all copies of vertices in the fixed order. All copies of $u\in Q$ are embedded at once using only
sets of auxiliary color $u$. Whenever $H$ is chosen as a child of $P$,
require
\begin{equation}\label{eq:fresh}
 (H\triangle P)\cap(P\triangle A)=\varnothing.
\end{equation}
This implies
\begin{equation}\label{eq:disjoint-support}
 H\triangle A=(P\triangle A)\mathbin{\dot\cup}(H\triangle P).
\end{equation}
Thus all copies of $u$ have distance exactly $s(u)$ from $A$.

Fix a nonroot original vertex $v$, let $u$ be its parent, and suppose
all copies of $u$ are already embedded. Put $j=w(uv)$. Since $s(u)+j\le R$,
every parent set $P$ has at least $cn^j/(4q)$ neighbors of color $v$.
At most
\[
 R\binom n{j-1}=O(n^{j-1})
\]
of them change an element in $P\triangle A$, since $H$ is determined by $H\triangle P$ once $P$ is given. For large $n$,
each parent therefore has at least $cn^j/(8q)$ eligible candidates satisfying
\eqref{eq:fresh}.

For each distinct candidate $H$, assign it independently and uniformly
to one of its eligible parents. By \eqref{eq:disjoint-support}, any such
parent $P$ satisfies $P\triangle A\subseteq H\triangle A$. Since
$|H\triangle A|\le R$, a candidate has at most $2^R$ eligible parent
sets. Every parent receives an expected number of candidates at least
$cn^j/(8q2^R)$. Its received count is a sum of independent Bernoulli
variables, and Chernoff's inequality gives
\[
 \Pr\left(\text{it receives fewer than }\frac{c}{8q2^{R+1}}n
       \text{ candidates}\right)=e^{-\Omega(n)}.
\]
Choose $0<\beta\le\min\{1,c/(8q2^{R+2})\}$ and take any integer
$1\le b\le\beta n$. There are at most $qn^{q-1}$ parents in $Q(y,b)$.
A union bound therefore shows that an allocation giving every parent
at least $b$ candidates exists. Pick $b$ candidate for each parent.

Perform this allocation once for each original nonroot vertex $v$,
with all its copies allocated in a single batch. No set of color $v$
has been used at an earlier stage: $v$ has a unique original parent,
and different original vertices use different colors. The allocation
makes copies of $v$ distinct. This proves injectivity, and
\eqref{eq:disjoint-support} establishes the radius bound.

The estimates and the lower bound on $n$ are uniform over the prescribed
root $A$ and over host graphs satisfying the stated degree bounds.
In particular, the root may be selected using an earlier part of an
embedding. (This will be important in the proof of Lemma \ref{lem:gap}!) For $q\ge2$ one may take
\begin{equation}\label{eq:beta}
 \beta=\min\left\{1,\frac{c}{q\,2^{R+5}}\right\}.
\end{equation}
\end{proof}

\begin{lemma}\label{lem:gap}
Let $S$ be a fixed saturated tree poset of height $h+1\ge2$, with
 levels $S_0,\ldots,S_h$. Fix $x\in S$, put $s=|S|$, and define integers
\begin{equation}\label{eq:scales}
 K_0=1,\qquad K_{a+1}=2(s-1)K_a+s\quad(0\le a\le h),
 \qquad Z=K_{h+1}+2.
\end{equation}
Suppose there are nonempty, pairwise disjoint host families
$\V_0,\ldots,\V_h$ and graphs $B_i$ with the properties in
Proposition~\ref{prop:extraction}, for this $Z$ and some fixed $c>0$.

Then, for some $\delta>0$ and all sufficiently large $n$, $\cup_{i=0}^h\V_i$ contains a weak copy of
$S(x,\lfloor\delta n\rfloor)$, with level $i$ represented in $\V_i$.
The constants can be chosen uniformly over $x$ and all type sequences.
\end{lemma}

\begin{proof}
There are $h+1$ disjoint intervals $(K_0,K_1],\ (K_1,K_2],\ldots,(K_h,K_{h+1}]$,
but at most $h$ bounded types among the graphs $B_1,\ldots,B_h$.
Choose $a\in\{0,\ldots,h\}$ for which $(K_a,K_{a+1}]$ contains none
of these bounded types. Call a relation of $S$ \emph{short} if it has bounded type $j\le K_a$, and
\emph{long} otherwise. Because the chosen interval is empty, every long
relation has either bounded type $j>K_{a+1}$ or type $+$. Its minimum
degree is at least $cn^{K_{a+1}+1}$.
For type $+$ this follows from
$Z-1=K_{h+1}+1\ge K_{a+1}+1$.

The idea is to break up $S$ into components with only short edges and embed them using Lemma~\ref{lem:bounded}. Then for long edges use the fact that the minimum degrees in the corresponding $B_i$ are very large to obtain that the centers of the balls used by the embeddings guaranteed by Lemma~\ref{lem:bounded} can be placed  so far from each other that the balls are pairwise disjoint and therefore the whole embedding is injective.

\smallskip

Delete from $H(S)$ all edges of long relations. Each subposet $Q$ of $S$ corresponding to a component of the remainder lies in a maximal consecutive block of levels $p,\ldots,q$ whose
intervening relations are all short. $Q$ is itself
saturated, of height $q-p+1$. Indeed, saturation of $S$ allows every
vertex of $Q$ to be extended downward to level $p$ and upward to level
$q$ along a maximal chain of $S$. These chains use only short
edges and therefore stay in $Q$. Consequently every maximal chain of
$Q$ starts on level $p$, ends on level $q$, and has $q-p+1$ elements.
After shifting level indices, the natural levels of $Q$ are
$Q\cap S_p,\ldots,Q\cap S_q$, with host families
$\V_p,\ldots,\V_q$. Every edge of $B_{p+1},\dots,B_q$ uses a bounded size difference at most
$K_a$, so Lemma~\ref{lem:bounded} applies in precisely this setting.

Root $S$ at $x$, and for every component $Q$ give $Q$ the unique root $y \in Q$ closest to $x$.
Put $L=(s-1)K_a$. Remember that the weight of an edge in $B_i$ is the type of $B_i$.
As all its edges are short, every weighted root-to-vertex path in $Q$ has total weight at most $L$.
Lemma~\ref{lem:bounded} embeds $Q(y,b)$ from any prescribed root set $A\in \V_{\ell(y)}$
inside $\B_L(A)$, for $b=\lfloor\delta n\rfloor$
and sufficiently small fixed $\delta>0$.
This choice can be made uniformly. For example, since $|Q|\le s$ and
its weighted radius is at most $(s-1)Z$, \eqref{eq:beta} permits
\begin{equation}\label{eq:delta}
 \delta=\min\left\{\frac12,\frac{c}{s\,2^{(s-1)Z+6}}\right\}.
\end{equation}
Isolated components need no restriction. There are only finitely many
rooted component types and bounded weight assignments under
consideration, so a common lower bound on $n$ also suffices.

Deleting the corresponding long edges from $S(x,b)$ leaves components
of the form $Q(y,b)$. Indeed, every component root is either the root
of $S(x,b)$ or a child reached through a long edge; every remaining
original edge still supplies exactly $b$ children away from its
component root. Joining these components to a root in the direction inherited from the deleted long edge gives a rooted tree.

Now we start embedding $(x,b)$. Fix an ordering $x=y_1,y_2,\dots,y_k$ of the roots of all components $Q$ such that $\dist_{H(S)}(x,y_i)\le \dist_{H(S)}(x,y_j)$ for all $i\le j\le k$. A root $y_i$ should have $b^{\dist_{H(S)}(x,y_i)}$ copies. We intend to embed them such that all image sets $A^\alpha_i$ ($1\le i\le k$, $1\le \alpha \le b^{\dist_{H(S)}(x,y_i)}$) of all roots have 
pairwise Hamming distance greater than $2L$. Then the balls $\B_L(A^\alpha_i)$  are
pairwise disjoint. Start at any set in $\V_{\ell(x)}$, and embed the blow-up of the entire
short-edge component of $x$ using Lemma~\ref{lem:bounded}.

Suppose we have embedded all the blow-ups of all components with roots $y_1,\dots,y_{i-1}$. Then for $y_{i}$ there is a unique  $v\in S$ a neighbor of $y_i$ in $H(S)$ that lies on the path from $y_i$ to $x$. Then, by definition of the roots and the order $y_1,\dots,y_k$, all required copies of $v$ in $S(x,b)$ have already been embedded, and $y_iv$ is a long edge of $S$. Writing
$d=\dist_{\HH(S)}(x,v)$, in $S(x,b)$ there are
$b^d=\Theta(n^d)$ copies of $v$, and each has $b$ children that are copies of $y_i$.
All these $b^{d+1}=\Theta(n^{d+1})$ children are roots of new
components. The parent copies may lie in fewer than $b^d$ old
components, since an old component can contain several copies of $v$.
We attach the new components \emph{one by one}. At each choice the condition of having Hamming distance greater than $2L$ is imposed against every component root selected
so far, including all roots already selected for copies of $y_i$.

Suppose a new component must be attached by the long edge $y_iv$ to an already
embedded parent set $P$, the image of a copy of $v$. The corresponding graph $B_j$ provides at least
$cn^{K_{a+1}+1}$ candidate root sets, by its minimum degree condition, in the required
containment direction. The total number of already selected roots is at most
\begin{equation}\label{eq:total-size}
 |S(x,b)|=\sum_{u\in S}b^{\dist_{H(S)}(x,u)}
 \le sb^{s-1}\le sn^{s-1}.
\end{equation}
Each forbids at most
$|\mathcal B_{2L}(A^\alpha_i)|=\sum_{i=0}^{2L}\binom ni=O(n^{2L})$
candidates, so the total number of forbidden sets is $O(n^{s-1+2L})$, whereas
\[\label{eq:exponent-gap}
 K_{a+1}+1=2(s-1)K_a+s+1=s+2L+1.
\]
Consequently a candidate $A$ outside all the forbidden balls exists.

Choose $A$ as the new component root and apply Lemma~\ref{lem:bounded}
inside $\B_L(A)$. This ball is disjoint from every previously
used component ball, so the component embedding uses no previously
chosen set. 
Lemma~\ref{lem:bounded} is uniform over the prescribed root, so this
adaptive choice causes no additional condition.

Repeat for every required long child. The bound
\eqref{eq:total-size} remains valid throughout. Embeddings are injective
within components by Lemma~\ref{lem:bounded}, and between components
by disjointness of their balls. Short edges are respected by the local
embeddings, and long edges by the choice of component roots.

Finally, $s,Z$ are independent of $x$, and \eqref{eq:delta} is uniform
over $x$ and the type sequence. The remaining lower bounds on $n$ can
be made uniform by taking a maximum over finitely many possibilities.
\end{proof}

To  prove Theorem~\ref{thm:main}, we will need to show that the blow-up of a tree poset is a subposet of the blow-up of a saturated tree poset. This is mainly a result of Bukh from his paper \cite{Bukh}. For this reason we just sketch the proof of the following lemma. A poset $(P,\preceq)$ is a \textit{strong subposet} of $(Q,\preceq')$ if there exists an injection $\iota: P\rightarrow Q$ with $p\preceq p'$ if and only if $\iota(p) \preceq' \iota(p')$. A family $\G$ of sets is a \textit{strong copy} of $(P,\preceq)$ if there exists a bijection $\iota: P\rightarrow \G$ such that $p\preceq p'$ if and only if $\iota(p)\subseteq \iota(p')$. 

\begin{lemma}\label{satu}\
\begin{enumerate}
    \item 
    \cite[Lemma~5]{Bukh} For any tree poset $T$ there exists a saturated tree poset $S$ with $h(T)=h(S)$ such that $T$ is a strong subposet of $S$.
    \item 
    If $x\in T\subseteq S$ for tree posets as in part (i), then $T(x,b)$ is a strong subposet of $S(x,b)$ for any $b\ge 1$.
\end{enumerate}
\end{lemma}

\begin{proof}[Sketch of proof]
    Consider the level partition $Q_0,Q_1,\dots Q_{h-1}$ of $T$ and write $\ell_T(u)=i$ if $u\in Q_i$ as before. For every arc $\overrightarrow{uv}$ of $\overrightarrow{H}(T)$ subdivide the arc by $\ell_T(v)-\ell_T(u)-1$ vertices and if $u\in T$ is a maximal vertex with $\ell_T(u)<h-1$, then append an upward directed path with $h-1-\ell_T(u)$ vertices to $u$. The resulting poset $S$ is clearly a tree poset with $h(S)=h(T)$, and $\ell_T(u)=\ell_S(u)$ for all $u\in T$. Also, by construction, for any pair $u,u'\in T$, we have $\dist_{H(T)}(u,u')\le \dist_{H(S)}(u,u')$.

    To see part (2), one uses this distance property. We fix an order $x=u_1,u_2,\dots,u_{|T|}$ of $T$ such that $i<j$ implies $\dist_{H(T)}(u_i,y)\le\dist_{H(T)}(u_j,y)$. We embed copies of $u_i$ in $T(y,b)$ into copies of $u_i$ in $S(y,b)$. If all copies of all $u_i$ have been embedded for all $i<j$, then let $u_{i_j}$ be the neighbor of $(u_j$ on the $u_jx$-path in $H(T)$. The edge $u_{i_j}u_j$ of $T$ is replaced in $S$ by a path of length $\dist_{H(S)}(u_{i_j},u_j)\ge1.$
    For every already embedded copy of $u_{i_j}$, the corresponding rooted path in $(x,b)$ has $b^{\dist_{H(S)}(u_{i_j},u_j)}\ge b$ distinct terminal copies of $u_j$. Moreover, the sets of terminal copies arising from distinct copies of $u_{i_j}$ are disjoint. We may therefore select $b$ of them for every copy of $u_{i_j}$, obtaining the required copies of $u_j$. Repeating this inductively gives a strong embedding of $T(x,b)$ into $S(x,b)$.
\end{proof}

\begin{proof}[Proof of Theorem~\ref{thm:main}]
If $h(T)=1$, the tree poset has one vertex, and the conclusion is
immediate for sufficiently large $n$. Suppose $h(T)=h+1\ge2$.

By Lemma~\ref{satu}, it suffices to prove the assertion for the
saturated tree poset $S$.
Choose $Z$ from \eqref{eq:scales} for this $S$ and $h$. Apply
Proposition~\ref{prop:extraction} to $\F$, obtaining compatible host
graphs with a constant $c=c(\eps,h,Z)>0$. Lemma~\ref{lem:gap} embeds
$S(x,\lfloor\delta n\rfloor)$ for some fixed $\delta>0$, and then 
$T(x,\lfloor\delta n\rfloor)\subseteq S(x,\lfloor\delta n\rfloor)$.

The saturated extension $S$ and the parameters $h,|S|,Z$ depend only on $T$.
The constant of Proposition~\ref{prop:extraction} depends only on $\eps,h,Z$, and the constants of Lemma~\ref{lem:bounded} and Lemma~\ref{lem:gap} are uniform over the finitely many original roots $x\in T$.
This gives $\delta$ and $n_0$ depending only on $T,\eps$, as required.
\end{proof}

\section{Corollaries}

In this section we gather some known and new consequences of Theorem~\ref{thm:main}. First, 
let $\Forb(n,T)$ denote the collection of weak $T$-free subfamilies of
$2^{[n]}$. As all subfamilies of a $T$-free family are $T$-free, we immediately obtain $|\Forb(n,T)|\ge 2^{\La(n,T)}=2^{(h(T)-1+o(1))\binom{n}{\lfloor \frac{n}{2}\rfloor}}$. The next theorem was already proved in \cite{BGW} and in a stronger form in \cite{Jetal}, but it is immediately implied by Theorem~\ref{thm:main} and a result from \cite{PT}.

\begin{theorem}\label{thm:counting}
For every fixed tree poset $T$ of height $h$,
\[
    |\Forb(n,T)|
    =2^{(h-1+o(1))\binom{n}{\lfloor \frac{n}{2}\rfloor}}.
\]
\end{theorem}

\begin{proof}
Theorem~1.5 of Patk\'os and Treglown~\cite{PT} states that if the linear
blow-up conjecture holds for a tree poset $T$ and at least one root
$x\in T$, then
\[
    |\Forb(n,T)|=2^{(h(T)-1+o(1))\binom{n}{\lfloor \frac{n}{2}\rfloor}}.
\]
Theorem~\ref{thm:main} establishes their hypothesis for every
tree poset and, in fact, for every choice of the root.  The theorem is
therefore an immediate consequence of their result.
\end{proof}

\medskip

Next we consider the random version of the forbidden subposet problem. Let $\Pp(n,p)$ be the random subfamily of $2^{[n]}$ obtained by retaining
every set independently with probability $p$.  Write
$\La(\Pp(n,p),T)$ for the largest size of a $T$-free subfamily of
$\Pp(n,p)$. The next theorem was proved in a stronger form, but with completely different proof in \cite{Jetal}.

\begin{theorem}\label{thm:random}
Let $T$ be any fixed tree poset of height $h\geq 2$.  If
$p=p(n)$ satisfies $pn\to\infty$, then, with high probability,
\[
  \La(\Pp(n,p),T)
  =\bigl(h-1+o(1)\bigr)p\binom{n}{\lfloor \frac{n}{2}\rfloor}.
\]
\end{theorem}

The proof uses the two-stage graph-container algorithm of~\cite{PT} (which in turn originates in the proof for the case $T$ being a chain in \cite{BMT}).  The only part of the argument that was
specific to trees of radius at most two was the first-stage supersaturation
lemma.  The following observation supplies its replacement.

\begin{lemma}\label{lem:quadratic}
For every fixed tree poset $T$ of height $h$, every $x\in T$, and every
$\varepsilon>0$, there is a constant $c=c(T,x,\varepsilon)>0$ such that every
$\A\subseteq 2^{[n]}$ with
\[
    |\A|\geq (2h-2+\varepsilon)\binom{n}{\lfloor \frac{n}{2}\rfloor}
\]
contains a weak copy of $T(x,cn^2)$.
\end{lemma}

\begin{proof}
Let $T^+$ be obtained from $T$ by subdividing every edge of $H(T)$ once,
orienting both new edges consistently with the original edge.  The Hasse
diagram of $T^+$ is again a tree.  Every directed path with $h-1$ edges in
$T$ becomes a directed path with $2h-2$ edges in $T^+$, and contracting the
subdivision vertices maps every directed path of $T^+$ to one in $T$.
Consequently, $h(T^+)=2h-1$.
Regard $x$ as a vertex of $T^+$.  By
Theorem~\ref{thm:main}, $\A$ contains
$T^+(x,\delta n)$ for some $\delta>0$.

If $u\in T$ is at distance $r$ from $x$ in $H(T)$, then it is at distance
$2r$ from $x$ in $H(T^+)$.  Hence $T^+(x,d)$ contains $d^{2r}$ copies of
$u$.  Moreover, along every original edge, the two successive $d$-fold
branchings combine into one $d^2$-fold branching.  Keeping only the
vertices corresponding to the original vertices of $T$ therefore gives
the containment $T(x,d^2)\subseteq T^+(x,d)$.
Taking $d=\delta n$ proves the lemma with, for example,
$c=\delta^2/2$ for all sufficiently large $n$.
\end{proof}

\begin{proof}[Proof of Theorem~\ref{thm:random}]
Fix $x\in T$ and $\varepsilon>0$. For $\varepsilon/4$, Theorem~\ref{thm:main} and Lemma \ref{lem:quadratic} yield $\delta$ and $c$, respectively. Let $\eta=\min\{\delta,c\}$.  We describe the two-stage container lemma
obtained by running the algorithm from~\cite{PT}. We fix an order $\mathbb O$ of sets in $2^{[n]}$, an order $\mathbb T=t_1,t_2,\dots,t_{|T|}$ of elements of $T$ that starts with $x=t_1$ and for every $i$ $T[t_1,\dots,t_i]$ is a tree poset, and for every $d\ge 1$ we fix an order $\mathbb T^d$ of weak copies of $T(x,d)$ in $2^{[n]}$. The input of the algorithm is a $T$-free family $\F\subseteq 2^{[n]}$. We start by setting $\G^0=2^{[n]},\Hh^0=\emptyset$.

Then for every $i$ if $\G^{i-1},\Hh^{i-1}$ are defined, then in Step $i$, we consider the largest $d$ such that $\G^{i-1}$ contains a copy of $T(x,d)$ and fix the copy $\mathcal T$ that comes first in $\mathbb T^d$. If the set $G_x$ corresponding to $x$ in $\mathcal T$ is not in $\F$, then we let $\G^i=\G^{i-1}\setminus \{G_x\},\Hh^i=\Hh^{i-1}$ and proceed to Step $i+1$.

If $G_x\in \F$, then we start building a copy of $T$ in $\mathcal T\cap \F$ in the fixed order $\mathbb T$. Suppose we have found $G_x=G_{t_1},\dots,G_{t_{j-1}}$ in $\mathcal T\cap \F$ such that these form a copy of $T[t_1,\dots,t_{j-1}]$ and let $t_\ell$ be the parent of $t_j$, i.e. the neighbor of $t_j$ on the $t_jx$-path in $H(T)$. Consider the $d$ neighbors of $G_{t_\ell}$ in $\mathcal T$ corresponding to $t_j$ in the order $\mathbb O$. If any of them belongs to $\F$, then let $G_{t_j}$ be the first one in $\mathbb O$, remove all neighbors preceding $G_{t_j}$ from $\G^{i-1}$ and proceed to find $G_{t_{j+1}}$. If no neighbor belongs to $\F$, then we remove all neighbors and $G_{t_1},\dots,G_{t_{j-1}}$ from $\G^{i-1}$ and whatever remains is $\G^i$ and set $\Hh^i=\Hh^{i-1}\cup \{G_{t_1},G_{t_2},\dots,G_{t_{j-1}}\}$. As $\F$ is $T$-free, there will be a $j$ for which $G_{t_j}$ is not defined. So either we just remove $G_x$ from $\G^{i-1}$ or we move at most $|T|-1$ sets from $\G^{i-1}\cap \F$ to $\Hh^i$ and remove further at least $d$ sets from $\G^{i-1}$. This implies that $\Hh^{i}\subseteq \F$ for all $i$. We terminate the algorithm once $d$ drops below $\eta n$. Let this happen at step $i_2$.

We analyze the algorithm in two stages. Let $i_1$ be the first index when $d$ drops below $\eta n^2$ and let $\Hh_1=\Hh^{i_1}$. By the above, we included new sets in $\Hh_1$ in at most $2^n/(\eta n^2)$ steps, so $|\Hh_1|=O_{T,\eta}(2^n/n^2)$. By Lemma 
~\ref{lem:quadratic}, $|\G^{i_1}|\le (2h-2+\varepsilon/4)\binom{n}{\lfloor \frac{n}{2}\rfloor}$. Similarly, in the next and final stage, there are at most $|\G^{i_1}|/(\eta n)$ steps when new sets are included into $\Hh^i$. Therefore letting $\Hh_2=\Hh^{i_2}\setminus \Hh_1$ we have $|\Hh_2|=O_{T,\eta}(\binom{n}{\lfloor \frac{n}{2}\rfloor}/n)$  and $\Hh_1,\Hh_2\subseteq \F$. Also, by Theorem~\ref{thm:main}, for $g(\Hh_1\cup \Hh_2):=\G^{i_2}$ we have $|g(\Hh_1\cup \Hh_2)|\le (h-1+\varepsilon/4)\binom{n}{\lfloor \frac{n}{2}\rfloor}$ and as we never removed sets of $\F$ without moving them to $\Hh^i$, we have $\F\subseteq \Hh_1\cup \Hh_2\cup g(\Hh_1\cup \Hh_2)$.
Since the algorithm worked with fixed orders $\mathbb O$, $\mathbb T$, $\mathbb T^d$, it follows as in~\cite{PT} (and in \cite{BMT}) that although $\Hh_1$ and $\Hh_2$ may be the same for different $T$-free families $\F$, $g(\Hh_1\cup \Hh_2)$ depends only on $\Hh_1$ and $\Hh_2$ and not on $\F$. Also, $\G^{i_1}$ depends only on $\Hh^{i_1}$, so $\Hh_2$ is a subset of a set depending on $\Hh_1$ of size at most $C_2\binom{n}{\lfloor \frac{n}{2}\rfloor}/n$.

To bound $\mathbb P(\exists \F\subseteq \Pp(n,p) ~\text{$T$-free with}\ |\F|\ge p(h-1+\varepsilon)\binom{n}{\lfloor \frac{n}{2}\rfloor})$, we bound the probability of the event that for some $\Hh_1$ and $\Hh_2$, we have $\Hh_1\subseteq \Pp(n,p),\Hh_2\subseteq \Pp(n,p)$ and $|\Pp(n,p)\cap g(\Hh_1\cup\Hh_2)|\ge p(h-1+\varepsilon/2)\binom{n}{\lfloor \frac{n}{2}\rfloor}$. Note that $pn\rightarrow \infty$ implies $|\Hh_1|,|\Hh_2|=o(p\binom{n}{\lfloor \frac{n}{2}\rfloor})$ and so a $T$-free family of size $p(h-1+\varepsilon)\binom{n}{\lfloor \frac{n}{2}\rfloor}$ should contain at least $p(h-1+\varepsilon/2)\binom{n}{\lfloor \frac{n}{2}\rfloor}$ sets from $g(\Hh_1\cup \Hh_2)$. As $\Hh_1,\Hh_2,g(\Hh_1\cup \Hh_2)$ are pairwise disjoint, for fixed $\Hh_1$ and $\Hh_2$, these are independent events, and by the Chernoff bound the probability of the last one is at most $e^{-\varepsilon^2p\binom{n}{\lfloor \frac{n}{2}\rfloor}/(100h^2)}$ as $\mathbb E|\Pp(n,p)\cap g(\Hh_1\cup\Hh_2)|\le p(h-1+\varepsilon/4)\binom{n}{\lfloor \frac{n}{2}\rfloor}$. Therefore, taking the union bound for each possible $\Hh_1,\Hh_2$, we obtain the upper bound
\[
\sum_{a\leq C_12^n/n^2}\binom{2^n}{a} p^a\sum_{b\leq C_2\binom{n}{\lfloor \frac{n}{2}\rfloor}/n}
     \binom{(2h-2+\varepsilon/4)\binom{n}{\lfloor \frac{n}{2}\rfloor}}{b}p^be^{-\varepsilon^2p\binom{n}{\lfloor \frac{n}{2}\rfloor}/(100h^2)}
\]
For the first sum, we have
\[
 \sum_{a\leq C_12^n/n^2}\binom{2^n}{a} p^a
 \leq
 \exp\!\left(O\!\left(\frac{2^n}{n^2}\log n\right)\right)
 =\exp(o(p\binom{n}{\lfloor \frac{n}{2}\rfloor})),                            
\]
because $\binom{n}{\lfloor \frac{n}{2}\rfloor}=\Theta(2^n/\sqrt n)$ and
 $\frac{2^nn^{-2}\log n}{p\binom{n}{\lfloor \frac{n}{2}\rfloor}}
 =O\!\left(\frac{\log n}{pn\sqrt n}\right)=o(1)$. For the second sum, we have
\[
 \sum_{b\leq C_2\binom{n}{\lfloor \frac{n}{2}\rfloor}/n}
     \binom{(2h-2+\varepsilon/4)\binom{n}{\lfloor \frac{n}{2}\rfloor}}{b}p^b
 \leq
 \exp\!\left(O\!\left(\frac{\binom{n}{\lfloor \frac{n}{2}\rfloor}}{n}\log(2+pn)\right)\right)
 =\exp(o(p\binom{n}{\lfloor \frac{n}{2}\rfloor})), 
\]
as $\log(2+pn)/(pn)=o(1)$.

As the negative exponent in the probability that $|g(\Hh_1\cup \Hh_2)\cap \Pp(n,p)|\ge p(h-1+\varepsilon/2)\binom{n}{\lfloor \frac{n}{2}\rfloor}$ is $\Omega(p\binom{n}{\lfloor \frac{n}{2}\rfloor})$, we get, with high probability, every
$T$-free subfamily of $\Pp(n,p)$ has size at most
$(h-1+\varepsilon)p\binom{n}{\lfloor \frac{n}{2}\rfloor}$.

For the lower bound, the union of the middle $h-1$ levels is $T$-free and
has size $(h-1+o(1))\binom{n}{\lfloor \frac{n}{2}\rfloor}$.  Its intersection with $\Pp(n,p)$ has size
$(h-1-o(1))p\binom{n}{\lfloor \frac{n}{2}\rfloor}$ with high probability by Chernoff's inequality.  Since
$\varepsilon>0$ was arbitrary, the result follows.
\end{proof}

\medskip

The final consequence of Theorem~\ref{thm:main} we consider is about maximal anti-Ramsey numbers of tree posets. Following the graph-theoretic problem of Burr, Erd\H os, Graham, and S\'os \cite{BEGS}, in \cite{LPW} for non-negative integers $n,m$ and a poset $P$ with $\La(n,P)<m\le 2^n$, the maximal anti-Ramsey number $\ar(n,m,P)$ is the smallest integer $k$ such that for any $\F\subseteq 2^{[n]}$ of size $m$ there exists a coloring using $k$ colors with all copies of $P$ in $\F$ being rainbow (all elements receiving distinct colors). This is easily seen to be $\min\{\chi(G_P(\F)):\F\subseteq 2^{[n]},|\F|=m\}$, where $G_P(\F)$ is the graph with vertex set $\F$ and $F,F'\in \F$ being joined by an edge if and only if $\F$ contains a copy $\G$ of $P$ with $F,F'\in \G$. For a tree poset $T$ one can consider the family $\mathcal M_h$ of $h=h(T)$ middle layers, i.e. $\mathcal M_{n,h}=\cup_{i=1}^h\binom{[n]}{\lfloor \frac{n-h}{2}\rfloor +i}$. By symmetry, for two sets $F,F'\in \mathcal M_{n,h}$ whether $\{F,F'\}$ is an edge in $G_T(\mathcal M_{n,h})$ depends only on $|F\cap F'|$, so $G_T(\mathcal M_{n,h})$ is a generalized Johnson graph. As shown in \cite{LPW}, by known results on Johnson graphs, for every tree poset there exists an integer $m_T$ such that $\chi(G_T(\mathcal M_{n,h}))=\Theta(n^{m_T})$. It is conjectured in \cite{LPW} that for any $\varepsilon>0$ there exists $\delta>0$ such that $\ar(n,(h-1+\varepsilon)\binom{n}{\lfloor \frac{n}{2}\rfloor},T)\ge \delta n^{m_T}$. The authors prove their conjecture for tree posets $T$ that contain an element comparable to all other elements of $T$ and that have a so-called branching property. We claim that Theorem~\ref{thm:main} implies the conjecture for further tree posets. Instead of defining a not very natural property that would suffice to prove the statement, we just show an example.

\begin{figure}[ht]
\centering
\begin{tikzpicture}[scale=1.05,
  every node/.style={circle,draw,inner sep=1.8pt},
  every edge/.style={draw,thick}]
  \node (c) at (0,0) {$x$};
  \node (d) at (5,0) {$a$};
  \node (a) at (-1,1.4) {$b$};
  \node (b) at (1.5,1.4) {$c$};
  \node (e) at (4,1.4) {$d$};
  \node (x1) at (-1.8,2.8) {$e$};
  \node (x2) at (-0.2,2.8) {$f$};
  \node (y) at (2.75,2.8) {$z$};
  \path (c) edge (a) edge (b)
        (a) edge (x1) edge (x2)
        (b) edge (y)
        (d) edge (e)
        (e) edge (y);
\end{tikzpicture}
\caption{a tree poset with no element being comparable to all other elements}
\label{fig:T2}
\end{figure}
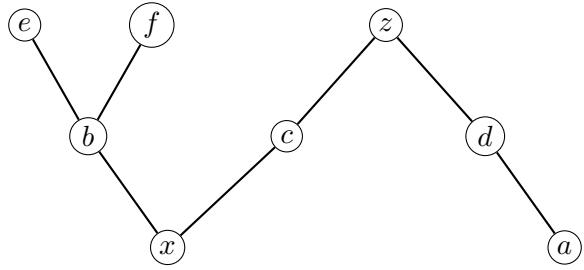
For the tree poset $T$ in Figure 1, we have $m_T=2$. Two sets $F,F'\in \binom{[n]}{\lfloor \frac{n-3}{2}\rfloor+i}$ are joined in $G_T(\mathcal M_{n,h})$ if and only if $|F\triangle F'|$ equals the graph distance of two elements with rank $i$ (ranks are from 1 to 3). Equivalently, two sets $F,F'$ of rank $i$ are joined if $|F\cap F'|=|F|-2$ or for $i=2,3$ if $|F\cap F'|=|F|-1$. So all layers have chromatic number $\Theta(n^2)$.

Now consider a family $\F\subseteq 2^{[n]}$ of size $(2+\varepsilon)\binom{n}{\lfloor \frac{n}{2}\rfloor}$. By Theorem~\ref{thm:main} it contains a copy $\mathcal T$ of the blow-up $T(x,\delta n)$ for some $\delta>0$. There are $(\delta n)^2$ sets $\mathcal Z$ in $\mathcal T$ corresponding to $z$. It is easy to see that for any two elements $z_1,z_2\in T(x,\delta n)$ corresponding to $z$, there exists a copy of $T$ in $T(x,\delta n)$ that contains $z_1,z_2$. Indeed, if $z_1,z_2$ have a common neighbor $c_1$, then with $c_3$, $z_3$, $d_3$, $a_3$ forming a path in $T(x,\delta n)$ and $c_1\neq c_3$, we have a bijection $\iota:T\rightarrow T(x,\delta n)$ with $\iota(e)=z_1,\iota(f)=z_2,\iota(b)=c_1,\iota(x)=x,\iota(c)=c_3,\iota(z)=z_3,\iota(d)=d_3,\iota(a)=a_3$. Finally, if the shortest path between $z_1,z_2$ in $T(x,\delta n)$ passes through $x$ and $c_1,c_2$ are the other two elements of the path with $c_i$ being a neighbor of $z_i$, then the mapping $\iota(x)=x,\iota(b)=c_1,\iota(e)=z_1,\iota(c)=c_2,\iota(z)=z_2$ can be extended to a bijection. This shows that $\mathcal Z$ forms a clique in $G_T(\mathcal T)$ and thus $\chi(G_T(\mathcal T))\ge (\delta n)^2$.

\bigskip

\noindent \textbf{AI declaration}: The author had a proof of Proposition \ref{prop:extraction} in the special case $h=3$ (as sketched in Remark \ref{rem}), the general proof was found by ChatGPT 5.6. All other mathematical ideas and proofs are due to the author. The first draft of the paper was written by ChatGPT which was then checked and re-written by the author who claims full responsibility of the content.

\end{document}